\documentclass[12 pt]{article}
\usepackage{stmaryrd}
\usepackage{times}
\usepackage{booktabs}
\usepackage{subfigure}
\usepackage{adjustbox}
\usepackage{rotating}
\usepackage{diagbox}
\usepackage{tabularx}
\usepackage{multirow, makecell}
\usepackage{rotating}
\usepackage{longtable}
\usepackage{supertabular}
\usepackage{pifont}
\usepackage{floatrow}
\usepackage{caption}
\usepackage{mathrsfs}
\usepackage[fleqn]{amsmath}
\usepackage{amsfonts,amsthm,amssymb,mathrsfs,bbding}
\usepackage{txfonts}
\usepackage{graphics,multicol}
\usepackage{graphicx}
\usepackage{makecell}
\usepackage{color}
\usepackage{caption}
\usepackage{ textcomp }
\usepackage{ stmaryrd }
\usepackage{indentfirst}
\usepackage{cite}
\usepackage{arydshln}
\usepackage{latexsym,bm}
\usepackage{enumerate}

\evensidemargin=\oddsidemargin\topmargin=-1.5cm

\newtheorem{lem}{Lemma}[section]
\newtheorem{defi}{Definition}[section]
\newtheorem{thm}{Theorem}[section]
\newtheorem{cor}{Corollary}[section]
\newtheorem{prop}{Proposition}[section]

\theoremstyle{definition}
\newtheorem{prob}{Problem}[section]

\newtheorem{re}{Remark}[section]
\newtheorem{clm}{Claim}[section]
\newtheorem{example}{Example}[section]

\begin{document}
\title{Signed graphs with symmetric spectra}
\author{Deqiong Li$\,^{1,}$\thanks{deqiongli@126.com}\qquad Qiongxiang Huang\,$^{2,}$\thanks{Corresponding author: huangqxmath@163.com}\\[0.2cm]
{\small $^1$ School of Mathematics and Statistics}\\
{\small Hunan  University of  Technology and Business, Changsha, Hunan 410205, China}\\
\\[0.01 cm]
{\small$^2$  College of Mathematics and Systems Science}\\
{\small Xinjiang University, Urumqi, Xinjiang 830046, China}\\}

\date{}
\maketitle
\begin{abstract}
A signed graph $\Gamma$ is spectrally symmetric  if, for each eigenvalue $\lambda$ of $\Gamma$, $-\lambda$ is also an eigenvalue of $\Gamma$ with the same multiplicity.  $\Gamma$ is said to be sign-symmetric if $\Gamma$ is switching isomorphic to   $-\Gamma$. It is well known  that     sign symmetry implies spectral symmetry,  but not vice versa. Bipartite signed graphs are  sign-symmetric. However,  there  exist spectrally symmetric non-bipartite signed graphs and sign-symmetric non-bipartite signed graphs. In this paper, we   provide   necessary and   sufficient  conditions for the spectrally symmetric signed graphs, as well as  for the sign-symmetric  signed graphs. Moreover, we give a classification of  sign-symmetric signed graphs and completely  determine one class of them that includes some known results.
\end{abstract}



\baselineskip=0.25in

\section {Introduction}
 Let $G=(V(G), E(G))$ be a simple connected graph.  A signed graph  $\Gamma=(G, \sigma)$ is a graph $G$ together with a \emph{sign function}   $\sigma: E(G)\rightarrow \{+1, -1\}$, where $G$ is called the \emph{underlying graph} of $\Gamma$ and $\Gamma$ is  a signed graph on $G$.  The vertex (resp. edge) set of $\Gamma$ is just the vertex (resp.
edge) set of its underlying graph, denoted by $V(\Gamma)$ (resp. $E(\Gamma)$). The signed graphs  contain unsigned graphs because $\Gamma=(G, \sigma)$ is exactly $G$ if $\sigma(e)=1$ for any edge $e\in E(\Gamma)$.
An edge $e\in E(\Gamma)$ is called \emph{positive} (resp. \emph{negative}) in $\Gamma$ if  $\sigma (e) = 1$ (resp. $\sigma (e) = -1$). For $U_1,U_2 \subseteq V(\Gamma)$, $\Gamma[U_1]$  denotes the subgraph of $\Gamma$  induced by $U_1$, and $E_\Gamma(U_1,U_2)$  denotes the set of  the edges of $\Gamma$ between $U_1$ and $U_2$.

Switching  $\Gamma=(G,\sigma)$ with respect to a  subset $U \subseteq V(\Gamma)$ means reversing the edge signs in the edge cut  $E_\Gamma(U, V(\Gamma)\setminus U)$,  the resulting graph is called  a \emph{switching} of $\Gamma$ denoted by $\Gamma^{U}=(G, \sigma^{U})$. The set
$[\Gamma]=\{\Gamma^U\mid U \subseteq V(\Gamma)\}$ forms a \emph{switching equivalence class} containing $\Gamma$, in which signed graphs  are  \emph{switching equivalent}.
Two signed graphs $\Gamma=(G, \sigma)$ and $\Gamma'=(G', \sigma')$  are \emph{isomorphic}, denoted by $\Gamma\cong\Gamma'$,
 if there is an  isomorphism between $G$ and $G'$ that  maintains edge signs. $\Gamma$ and $\Gamma'$  are called \emph{switching isomorphic}, denoted by $\Gamma \simeq \Gamma'$, if $\Gamma$ is isomorphic to a switching of  $\Gamma'$.
 An isomorphism from  $\Gamma$  to itself is  an \emph{automorphism} of $\Gamma$. All  automorphisms of $\Gamma$ consist of an  automorphism group  $Aut(\Gamma)$.

The adjacency matrix $A(\Gamma)$ of a signed graph $\Gamma$ is defined similarly
as an unsigned graph, with 1 for positive edges, $-1$ for negative edges, and 0 otherwise.
The set of eigenvalues of $A(\Gamma)$  with their multiplicities is called the \emph{spectrum} of  $\Gamma$ denoted by $\mathrm{Spec}(\Gamma)$.  It is clear that  there is a $(1, -1)$-diagonal matrix $D$ such that $A(\Gamma)=DA(\Gamma')D^\top$ if $\Gamma \sim\Gamma'$, and $\Gamma\simeq\Gamma'$ if and only if there is a $(0, 1, -1)$-permutation matrix $P$ such that $A(\Gamma)=PA(\Gamma')P^\top$.
So, switching equivalent (isomorphic) signed graphs have the same spectrum. The spectrum of a signed graph $\Gamma$ is called \emph{symmetric}  if, for each eigenvalue $\lambda$ of $A(\Gamma)$, $-\lambda$ is also an eigenvalue of $A(\Gamma)$ with the same multiplicity.

Let $-\Gamma= (G,-\sigma)$ be the signed
graph obtained from $\Gamma$ by reversing the signs of all edges.
A signed graph $\Gamma$ is said to be \emph{sign-symmetric} if $\Gamma\simeq-\Gamma$.
A sign-symmetric graph is spectrally symmetric but not the other way around \cite{Belardo}.
It is well known that  the spectrum of a graph is  symmetric if and only if it is a bipartite graph \cite{Cvetk}. Bipartite signed graphs are sign-symmetric, which means that they have symmetric spectra. In addition to  bipartite signed graphs, there are  non-bipartite signed graphs with symmetric spectra and there are  non-bipartite signed graphs with sign symmetry.   Therefore, the study of spectral symmetry  and  sign symmetry of signed graphs  focuses  on non-bipartite  signed  graphs.

 Akbari et al.    presented a  construction of completed signed graphs with sign symmetry  \cite{Akbaria}. In \cite{Belardo},
 Belardo et al.  constructed a complete signed graph with a symmetric spectrum but not sign-symmetric,  posing the following problem: Are there connected non-complete signed graphs whose spectrum is symmetric but they are not sign-symmetric?
 Ghorbani et al. in \cite{Ghorbani}    presented some infinite families  of signed graphs with
symmetric spectra but not sign-symmetric which answered the problem in \cite{Belardo}, and  gave a method for constructing  sign-symmetric signed graphs. Additionally, they enumerated all complete signed graphs of orders 4, 5, 6, 8, 9 with symmetric spectra and sign symmetry.
Stani\'{c} in \cite{Stanic} proved that the Cartesian product $\Gamma_1 \Box \Gamma_2$   and   corona $\Gamma_1 \circ K_1$  both have  symmetric spectra if $\Gamma_1$ and $\Gamma_2$ are two spectrally symmetric signed graphs. He furthermore constructed two infinite families of connected non-complete signed graphs being spectrally symmetric  but not  sign-symmetric.

Inspired by the above research,
in this paper we consider the spectral symmetry and sign symmetry of signed graphs.
Necessary and   sufficient  conditions are given  for spectrally symmetric signed graphs in Section 2 (see Theorem \ref{Mthm-1}), as well as for sign-symmetry of signed graphs in Section 3 (see Theorem \ref{C-p-4}). Moreover, we give a classification of the sign-symmetric signed graphs (see Proposition \ref{pro-1}).  In Section 4, we  give three constructing theorems (see Theorems \ref{thm-main-1}, \ref{ff-thm-2} and  \ref{ff-thm-3}) that completely determine one class of the  sign-symmetric signed graphs which includes some known results (see Theorems \ref{thm-10} and  \ref{thm-12}).

\section{Signed graphs with symmetric spectra}

 In this section, we will give a sufficient and  necessary condition for signed graphs with symmetric  spectra. First, we cite a well known algebraic condition for  spectral  symmetry (e.g.  Remark 3.2 in \cite{Ghorbani}), which   can  be easily verified.
\begin{thm}\label{spec-sym-cor-1}
The signed graph $\Gamma=(G,\sigma)$ has a symmetric  spectrum if and only if its  characteristic polynomial $\phi(\Gamma, x) =\sum_{i=0}^na_ix^{n-i}$  satisfies  $a_i =0$ for all odd indices $i$.
\end{thm}
Let $\Gamma'$ be  a subgraph of $\Gamma=(G,\sigma)$ (denoted by $\Gamma'\subseteq \Gamma$). The sign of  $\Gamma' $  is defined by $\sigma(\Gamma')=\prod_{e\in \Gamma'}\sigma(e)$. A cycle $C$ of $\Gamma$ is  positive if  $\sigma(C)=1$ and  negative otherwise.
A \emph{basic figure} of $\Gamma$ is  disjoint union of some cycles and edges.  Let $p(B)$  denote the number of components of the graph $B$.  The Sachs formula for  the signed graph version is presented as follows.

\begin{thm}[$\!\!$\cite{Belardo2}]\label{coeff-thm-1}
Let $\Gamma=(G, \sigma)$ be a signed graph with $n$ vertices and $\phi(\Gamma, x) =\sum_{i=0}^n a_ix^{n-i}$. Then
$a_i =\sum_{B\in\mathcal{B}_i}(-1)^{p(B)}2^{|C(B)|}\sigma(B)$, where $\mathcal{B}_i$ is the set of basic figures on $i$ vertices in $G$ and $C(B)$ is the set of cycles in $B$.
\end{thm}

 A  $2$-regular  subgraph $H$ of $\Gamma=(G, \sigma)$ is  the disjoint  union of some cycles in  $\Gamma$, which  is called  \emph{even} if  $|H|$ is even, and  \emph{odd} otherwise. Denote by
$\mathcal{H}$   the set of  $2$-regular subgraphs of  $\Gamma$. Let
$\mathcal{H}_0=\{H\in \mathcal{H}\mid |H|~\text{is~even}\}$, $\mathcal{H}_1=\{H\in \mathcal{H}\mid |H|~\text{is~odd}\}$,
$\mathcal{H}_0^+=\{H\in \mathcal{H}_0\mid \sigma(H)=1\} $, and $\mathcal{H}_0^-=\{H\in \mathcal{H}_0\mid \sigma(H)=-1\} $. Similarly,
$\mathcal{H}_1^+ $ and $\mathcal{H}_1^-$ denote the sets of positive and negative  $2$-regular subgraphs in $\mathcal{H}_1$, respectively.
 Obviously,  we have $\mathcal{H}=\mathcal{H}_0\cup\mathcal{H}_1$,  $\mathcal{H}_0=\mathcal{H}_0^+\cup\mathcal{H}_0^-$, and  $\mathcal{H}_1=\mathcal{H}_1^+\cup\mathcal{H}_1^-$.

A \emph{$k$-matching} of a  graph is a set of $k$ independent edges. The \emph{matching polynomial} of a graph $G$, denoted by $M(G, x)$,  is defined by
\begin{equation}\label{MM-eq-2}
M(G, x)=\sum_{i=0}^{m(G)}
(-1)^im_i(G)x^{n-2i},
\end{equation}
where $m(G)$ is the size of the largest matching in $G$, and $m_i(G)$ is the number of
$i$-matchings in $G$.
If $\Gamma'\subseteq\Gamma$, we write $\Gamma-\Gamma'$ for the signed graph induced by $V(\Gamma)\setminus V(\Gamma')$.
The characteristic polynomial and the matching polynomial of a signed graph  have the following  relationship.

\begin{thm}[$\!\!$\cite{Belardo2}]\label{MM-pro-1}
The characteristic polynomial of a signed graph $\Gamma=(G, \sigma)$  can be expressed as
$$\phi(\Gamma, x) = M(G, x) + \sum_{C\in \mathcal{H}}
\sigma(C)(-2)^{p(C)}M(G-C, x).$$
\end{thm}

\begin{thm}\label{Mthm-1}
Let $\Gamma=(G, \sigma)$ be a  signed graph. Then $\Gamma$ has a symmetric spectrum if and only if
\begin{equation}\label{MM-eq-4}
\sum_{C\in \mathcal{H}_1^+}
(-2)^{p(C)}M(G-C, x)=\sum_{C\in \mathcal{H}_1^-}
(-2)^{p(C)}M(G-C, x).
\end{equation}
\end{thm}
\begin{proof}
By Theorem \ref{MM-pro-1}, we have
$$\phi(\Gamma, x) = M(G, x) + \sum_{C\in \mathcal{H}_0}
\sigma(C)(-2)^{p(C)}M(G-C, x)+ \sum_{C\in \mathcal{H}_1}
\sigma(C)(-2)^{p(C)}M(G-C, x).$$
 By (\ref{MM-eq-2}), it is easy to see that each term of $M(G, x)$ contributes to a coefficient of   the form $a_{2i}x^{n-2i}$ in  $\phi(\Gamma,x)$. Moreover,
\begin{eqnarray*}
 \sum_{C\in \mathcal{H}_0}\sigma(C)(-2)^{p(C)}M(G-C, x) &=& \sum_{C\in \mathcal{H}_0}\sum_{j\ge0}(-2)^{p(C)}\sigma(C)(-1)^jm_j(G-C)x^{|G-C|-2j} \\
  &=& \sum_{C\in \mathcal{H}_0}\sum_{j\ge0}(-2)^{p(C)}\sigma(C)(-1)^jm_j(G-C)x^{|G|-(|C|+2j)}.
\end{eqnarray*}
Since  $C\in \mathcal{H}_0$,  by Theorem \ref{coeff-thm-1}, $\sigma(C)(-2)^{p(C)}(-1)^j$ $m_j(G-C)$ only contributes to some coefficient of the form $a_{2i}x^{n-2i}$ in $\phi(\Gamma,x)$.
Similarly,  we have
\begin{equation}\label{MC-eq-1}
\sum_{C\in \mathcal{H}_1}
\sigma(C)(-2)^{p(C)}M(G-C, x)=\sum_{C\in \mathcal{H}_1}\sum_{j\ge 0}\sigma(C)(-2)^{p(C)}(-1)^jm_j(G-C)x^{|G|-(|C|+2j)}.
\end{equation}
 $|C|+2j$ is odd because of $C\in \mathcal{H}_1$. According to  Theorem \ref{coeff-thm-1}, each  term $\sigma(C)(-2)^{p(C)}(-1)^j$ $m_j(G-C)$ in (\ref{MC-eq-1}) contributes to a coefficient  of the form $a_{2i+1}x^{n-(2i+1)}$ in $\phi(\Gamma,x)$. By Theorem \ref{spec-sym-cor-1}, 
 the spectrum of $\Gamma$ is symmetric if and only if  $\sum_{C\in \mathcal{H}_1}
\sigma(C)(-2)^{p(C)}M(G-C, x)=0$. It follows that $\Gamma$ is spectrally symmetric  if and only if (\ref{MM-eq-4}) holds.
\end{proof}

\begin{figure}[h]
  \centering
  \includegraphics[width=3cm]{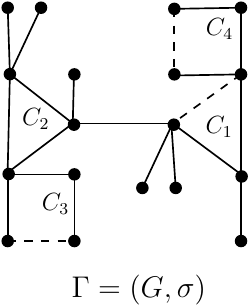}\\
  \caption{\small $\Gamma$ is spectrally symmetric  but  not sign-symmetric. Here and in the remaining figures  negative edges
are represented by dash lines.}\label{figure11}
\end{figure}
\begin{example}\label{ex-con-10}
In Fig.$\!$ \ref{figure11}, it can be verified that $G-C_1\cong G-C_2$ and $G-C_1-C_3\cong G-C_2-C_4$. Thus,  \eqref{MM-eq-4} holds for the signed graph $\Gamma$.  From Theorem \ref{Mthm-1},  $\Gamma$ has a symmetric spectrum. In fact,   $\mathrm{Spec}(\Gamma)=\{\pm 2.821,  \pm 2, \pm 1.792, \pm 1.414^{[2]}, \pm 1, \pm 0.737, \pm 0.537, 0^{[2]}\}$. Moreover, we  observe that $\Gamma$ is not sign-symmetric.
\end{example}

Theorem \ref{Mthm-1} states that $\Gamma=(G,\sigma)$ is  spectrally symmetric if  $G$ is a  bipartite graph, as bipartite graphs contain no 2-regular subgraphs.

\begin{cor}\label{l=l-cor-0}
Let $G$ be  a unicyclic graph and $\Gamma=(G, \sigma)$, then $\Gamma$ has a  symmetric spectrum if and only if $G$ has  an even cycle.
\end{cor}

\begin{cor}\label{l=l-cor-1}
Let $G$ be a bicyclic graph and $\Gamma=(G,\sigma)$. $\Gamma$ has  a symmetric spectrum if and only if  one of the following holds:\\
 (i) the cycles of $\Gamma$ are  even, or\\
 (ii) there are two odd cycles $C_1$ and $C_2$  such that $M(G-C_1,x)=M(G-C_2,x)$ and  $\sigma(C_1)\sigma(C_2)=-1$.
\end{cor}

Based on Corollary \ref{l=l-cor-1} (ii),
we can construct an infinite
family of non-bipartite bicyclic signed graphs with  symmetric spectra:
First, take   two  forests $F_1$ and $F_2$ with $\mathrm{Spec}(F_1)$ $=\mathrm{Spec}(F_2)$, which  implies that they have the same matching polynomial by Theorem \ref{MM-pro-1}. Next, construct
 a bicyclic signed graph $\Gamma=(G, \sigma)$ with two  odd cycles $C_1$ and $C_2$ of opposite signs such that  $G-C_1\cong F_1$ and $G-C_2\cong F_2$. It follows that
\begin{equation}\label{MM-eq-5}
M(G-C_1, x)= M(G-C_2, x).
\end{equation}
According to Corollary \ref{l=l-cor-1} (ii), $\Gamma$ is spectrally symmetric. Moreover, if   $F_1$ and $F_2$ are two non-isomorphic forests, then $\Gamma$ is  not sign-symmetric.
To explain the construction, we give the following example.

\begin{figure}[h]
  \centering
  \includegraphics[width=11cm]{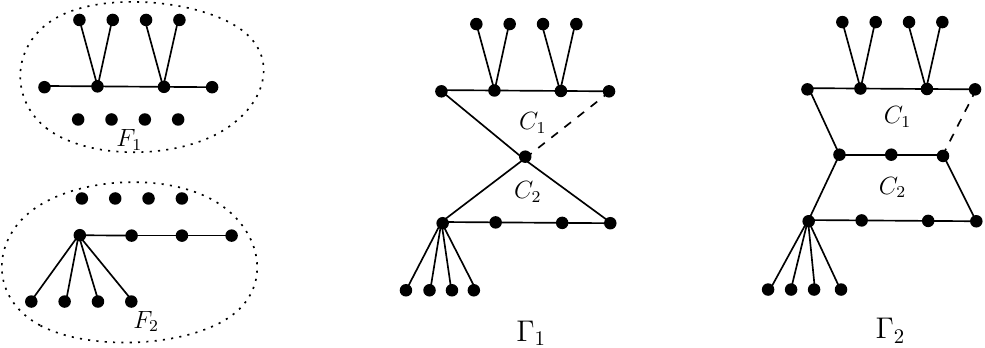}\\
  \caption{\small $\Gamma_1$ and $\Gamma_2$ are two  signed graphs that are spectrally symmetric but not sign-symmetric.}\label{figM}
\end{figure}
 \begin{example}\label{ex1}
The  forests $F_1$ and $F_2$  shown in Fig.$\!$ \ref{figM} have the same spectrum     $\mathrm{Spec}(F_1)=\mathrm{Spec}(F_2)=\{\pm 2.303, \pm 1.303, 0^{[8]}\}$,  the bicyclic signed graphs $\Gamma_1$ and $\Gamma_2$ satisfy (\ref{MM-eq-5}).  So $\Gamma_1$ and $\Gamma_2$ are  spectrally symmetric.
 We calculate that $\mathrm{Spec}(\Gamma_1)=\{\pm 2.711, \pm 2.303, \pm 1.702,$ $ \pm 1.303, \pm 0.8667, 0^{[7]}\}$ and $\mathrm{Spec}(\Gamma_2)=\{\pm2.624,\pm 2.369,\pm1.887,\pm1.461,\pm1.121,\pm0.744,$ $0^{[7]}\}$. Furthermore, $\Gamma_1$ and $\Gamma_2$ are not sign-symmetric  since $F_1$ and $F_2$ are not isomorphic.

\end{example}


\section{Signed graphs with  sign symmetry  }
 In this section, we will give a sufficient and  necessary condition for sign-symmetric signed graphs.

\begin{lem}[$\!$\cite{Zaslavsky1}]\label{thm-1}
Let $\Gamma$ and $\Gamma'$  be two signed graphs with the same underlying graph. Then $\Gamma\sim\Gamma'$ if and only if they have the same  positive cycles.
\end{lem}
\begin{lem}\label{FF-lem-2}
If $\Gamma\simeq\Gamma'$, then $\Gamma$ is sign-symmetric if and only if $\Gamma'$ is sign-symmetric.
 \end{lem}
 \begin{proof}
Since $\Gamma\simeq\Gamma'$, there is a $(0, 1, -1)$-permutation matrix $P$ such that $A(\Gamma)=PA(\Gamma')P^\top$.  Assume that $\Gamma$ is sign-symmetric. There exists a $(0, 1, -1)$-permutation matrix $P'$ such that $P'A(\Gamma)(P')^\top=-A(\Gamma)$. So we have $(P^\top P'P)A(\Gamma')(P^\top P'P)^\top=-A(\Gamma')$. It follows that $\Gamma'$ is sign-symmetric. Similarly, we obtain that $\Gamma$ is sign-symmetric if $\Gamma'$ is sign-symmetric.
 \end{proof}

By Lemmas \ref{thm-1} and \ref{FF-lem-2}, two switching equivalent (switching isomorphic) signed graphs  are identical when we consider the sign symmetry.


Let $\Gamma=(G, \sigma)$ be a signed graph.
Let $\mathcal{C}(\Gamma)$ (or simply $\mathcal{C}$ when no confusion arises) be   the set of cycles of  $\Gamma$. $\mathcal{C}$ can be decomposed into two parts: $\mathcal{C}_0=\{C\in \mathcal{C} \mid  |C|~ \text{is~ even}\}$ and $\mathcal{C}_1=\{C\in \mathcal{C}\mid |C|~ \text{ is ~odd}\}$.
Let $\mathcal{C}_0^+=\{C\in \mathcal{C}_0 \mid   \sigma(C)=1\}$ and $\mathcal{C}_0^-=\{C\in \mathcal{C}_0 \mid   \sigma(C)=-1\}$; $\mathcal{C}_1^+=\{C\in \mathcal{C}_1 \mid   \sigma(C)=1\}$ and $\mathcal{C}_1^-=\{C\in \mathcal{C}_1 \mid   \sigma(C)=-1\}$.
\begin{defi}\label{CC-de-1}
Let $\Gamma=(G, \sigma)$ be a  signed graph. An automorphism $\varphi\in Aut(G)$  is called a \emph{weak automorphism}  of  $\Gamma$ if $\varphi$ satisfies
$$\sigma(\varphi(C))=\left\{\begin{array}{ll}
-\sigma(C)& \text{ for~ any~ cycle}~  C\in \mathcal{C}_1,\\
 \sigma(C)& \text{ for~ any~ cycle }~ C\in \mathcal{C}_0.\\
 \end{array}\right.$$
\end{defi}
For a  bipartite signed graph, we have $\mathcal{C}_1^+=\emptyset=\mathcal{C}_1^-$.
 So, a  bipartite signed graph always has a weak automorphism $\iota\in Aut(G)$, where $\iota$ denotes the identity.
As $\varphi\in Aut(G)$, we have  $\varphi(\mathcal{C}_1)=\mathcal{C}_1=\mathcal{C}_1^-\cup \mathcal{C}_1^+$. Thus, $\varphi(\mathcal{C}_1^+)=\mathcal{C}_1^-$ if and only if   $\varphi(\mathcal{C}_1^-)=\mathcal{C}_1^+$. Similarly, $\varphi(\mathcal{C}_0^+)=\mathcal{C}_0^+$ if and only if $\varphi(\mathcal{C}_0^-)=\mathcal{C}_0^-$. It immediately follows the result.
\begin{lem}\label{lemma-33}
$\varphi\in Aut(G)$  is  a weak automorphism  of   $\Gamma=(G, \sigma)$ if  and only if $\varphi(\mathcal{C}_1^+)=\mathcal{C}_1^-$ and $\varphi(\mathcal{C}_0^+)=\mathcal{C}_0^+$ for some $\varphi\in Aut(G)$.
\end{lem}

Let $T$ be  a spanning tree of a graph $G$. Each edge $e\in E(G)\backslash E(T)$ forms a unique cycle  with some edges in  $T$. Such a cycle  is called a \emph{base cycle} of $G$ with respect to  $T$ and  denoted by $C_G(T,e)$.     The  base cycle of $\Gamma=(G, \sigma)$ is defined similarly and denoted by $C_{\Gamma}(T,e)$.
Let
 $\mathcal{B}(\Gamma, T)=\{C_{\Gamma}(T,e)\mid e\in E( \Gamma)\backslash E(T)\}$ be the set of base cycles of $\Gamma$,
 $\mathcal{B}^{+}(\Gamma, T)$ and $\mathcal{B}^{-}(\Gamma, T)$ be the sets of positive and negative base cycles, respectively.
Let
$\mathcal{B}_0(\Gamma, T)=\{C\in\mathcal{B}(\Gamma, T) \mid |C|~\text{is~ even}\}$ and  $\mathcal{B}_1(\Gamma, T)=\{C\in\mathcal{B}(\Gamma, T) \mid |C|~\text{is~ odd}\}$. Let $\mathcal{B}_0^+(\Gamma, T)=\{C\in\mathcal{B}_0(\Gamma, T) \mid \sigma(C)=1\}$ and $\mathcal{B}_0^-(\Gamma, T)=\{C\in\mathcal{B}_0(\Gamma, T) \mid \sigma(C)=-1\}$. $\mathcal{B}_1^+(\Gamma, T)$, $\mathcal{B}_1^-(\Gamma, T)$ are similarly defined.
Without causing confusion, these symbols are abbreviated as
 $\mathcal{B}(T)$, $\mathcal{B}^{+}(T)$, $\mathcal{B}^{-}( T)$, $\mathcal{B}_0(T)$, $\mathcal{B}_1( T)$, $\mathcal{B}_0^+(T)$, $\mathcal{B}_0^-(T)$, $\mathcal{B}_1^+(T)$ and $\mathcal{B}_1^-(T)$, respectively. One can easily verify that
 \begin{lem}\label{lemma-36}
 Let $T$ be a  spanning tree of $\Gamma=(G,\sigma)$, and $\varphi\in Aut(G)$  with $\varphi(T)=T'$. Then\\
(i) $\varphi(\mathcal{B}_1^+(T))=\mathcal{B}_1^-( T')$ if and only if  $\varphi(\mathcal{B}_1^-( T))=\mathcal{B}_1^+( T')$,\\
(ii) $\varphi(\mathcal{B}_0^+( T))=\mathcal{B}_0^+( T')$ if and only if $\varphi(\mathcal{B}_0^-( T))=\mathcal{B}_0^-( T')$.
\end{lem}


\begin{lem}\label{lemma-34}
Let $\Gamma=(G,\sigma)$ be a signed graph and $\varphi\in Aut(G)$ such that  $\varphi(T)=T'$ for a  spanning tree $T$ of $\Gamma$. Then
$\varphi(\mathcal{C}_1^+)=\mathcal{C}_1^-$ and $\varphi(\mathcal{C}_0^+)=\mathcal{C}_0^+$ if and only if   $\varphi(\mathcal{B}_1^+( T))=\mathcal{B}_1^-( T')$ and $\varphi(\mathcal{B}_0^+( T))=\mathcal{B}_0^+( T')$.
\end{lem}
\begin{proof}
Suppose that $\varphi(\mathcal{C}_1^+)=\mathcal{C}_1^-$ and $\varphi(\mathcal{C}_0^+)=\mathcal{C}_0^+$. Obviously, $\varphi(\mathcal{B}_1^+( T))=\mathcal{B}_1^-( T')$ and $\varphi(\mathcal{B}_0^+( T))=\mathcal{B}_0^+( T')$ because of $\varphi(T)=T'$.

Conversely, suppose that $\varphi(\mathcal{B}_1^+( T))=\mathcal{B}_1^-( T')$ and $\varphi(\mathcal{B}_0^+( T))$ $=\mathcal{B}_0^+( T')$. Any cycle $C\in \mathcal{C}_1^+$ can be uniquely generated by some base cycles in $\mathcal{B}(T)$, say,
\begin{equation}\label{CC-eq-11}
C =(C_{\Gamma}(T,e_1)\oplus\cdots\oplus C_{\Gamma}(T,e_{t_1}))\oplus(C_{\Gamma}(T,e_{t_1+1})\oplus\cdots\oplus C_{\Gamma}(T,e_{t})),
\end{equation}
where $\oplus$ represents set symmetry difference, and $C_{\Gamma}(T,e_1),\ldots, C_{\Gamma}(T,e_{t_1})$  are odd base cycles and $C_{\Gamma}(T,e_{t_1+1}),\ldots,C_{\Gamma}(T,e_{t})$ are even base cycles.
Clearly, we have $\varphi(C_{\Gamma}(T,e_i))=C_{\Gamma}(T',e_i')$ where $\varphi(e_i)=e_i'$ for $i=1, \ldots, t$.
It follows that
\begin{equation}\label{CC-eq-12}
\varphi(C)=\oplus_{i=1}^t\varphi(C_{\Gamma}(T,e_i))=
\oplus_{i=1}^tC_{\Gamma}(T',e_i').
\end{equation}
Furthermore, assume that  the first $t_1'\le t_1$ odd base cycles  are negative and the first $t_1''\le t-t_1$ even base cycles are negative in \eqref{CC-eq-11}. Since $C$ is positive, so $t_1'+t_1''$ is an even integer which implies that $t_1'-t_1''$ is  even.
According to our assumption,  we have $\sigma(C_{\Gamma}(T',e_i'))=-\sigma(C_{\Gamma}(T,e_i))$ for  $i=1, \ldots, t_1$, and  $\sigma(C_{\Gamma}(T',e_{i}'))=\sigma(C_{\Gamma}(T,e_i))$ for $i=t_1+1, \ldots, t$.  By Lemma \ref{lemma-36}, there are exactly $t_1-(t_1'-t_1'')$ negative base cycles  in $C_{\Gamma}(T',e_i')$ ($i=1, \ldots, t$) of \eqref{CC-eq-12}. Notice that $t_1$ must be an odd integer because $C$ is  odd. So, $t_1-(t_1'-t_1'')$ is  odd.  Thus, we have $\sigma(\varphi(C))=-1$.
 It follows that $\varphi(C)\in \mathcal{C}_1^-$, i.e., $\varphi(\mathcal{C}_1^+)=\mathcal{C}_1^-$.
For even cycle set of $\Gamma$,   similarly, we obtain that $\varphi(\mathcal{C}_0^+)=\mathcal{C}_0^+$.

It completes the proof.
\end{proof}
From Lemmas \ref{lemma-33} and  \ref{lemma-34}, we obtain immediately
\begin{thm}\label{cor-exchan-1}
Let $\Gamma=(G,\sigma)$ be a signed graph and $\varphi\in Aut(G)$ such that $\varphi(T)=T'$ for a  spanning tree $T$ of $G$. Then $\varphi$ is a weak automorphism of $\Gamma$  if and only if $\varphi(\mathcal{B}_1^+(T))=\mathcal{B}_1^-( T')$ and $\varphi(\mathcal{B}_0^+( T))=\mathcal{B}_0^+( T')$.
\end{thm}

\begin{lem}\label{conver-lem-1}
Let $\Gamma=(G,\sigma)$ and $\Gamma'=(G',\sigma')$ be two signed graphs. If there exits an isomorphism $\varphi:G\longrightarrow G'$ such that $\varphi(\mathcal{C}_1^+(\Gamma))=\mathcal{C}_1^-(\Gamma')$ and $\varphi(\mathcal{C}_0^+(\Gamma))=\mathcal{C}_0^+(\Gamma')$, then $\varphi(-\Gamma)\sim \Gamma'$, i.e., $\Gamma'\simeq-\Gamma$.
\end{lem}
\begin{proof}
Let $T$ be a spanning tree of  $G$. Then  $\varphi(T)=T'$ is  a spanning tree of $G'$. It is easy to see that there exists $\Gamma''=(G',\sigma'')\in [\Gamma']$  such that
\begin{equation}\label{CC-eq-18}
 \sigma''(\varphi(e))=-\sigma(e)
\end{equation}
 for any edge $e\in T$.
We will  distinguish some situations to prove that $\sigma''(e')=-\sigma(e)$ for each $e\in E(G)\backslash E(T)$, where $e'=\varphi(e)$.

Suppose that  $C_\Gamma(T,e)\in \mathcal{C}_1^+(\Gamma)$. By assumption, we have $\varphi(C_\Gamma(T,e))=C_{\Gamma'}(T',e')\in  \mathcal{C}_1^-(\Gamma')$ and so $C_{\Gamma''}(T',e')\in  \mathcal{C}_1^-(\Gamma'')$ by Lemma \ref{thm-1}. According to \eqref{CC-eq-18}, it follows that       $e'$ is a positive edge of $\Gamma''$ if $e$ is a negative edge of $\Gamma$, and $e'$ is a negative edge of $\Gamma''$ if $e$ is a positive edge of $\Gamma$. That is $\sigma''(e')=-\sigma(e)$.
Similarly,
we can verify  that $\sigma''(e')=- \sigma(e)$  when $C_\Gamma(T,e)\in \mathcal{C}_1^-(\Gamma)$.

Suppose that $C_\Gamma(T,e)\in \mathcal{C}_0^+(\Gamma)$. We have $\varphi(C_\Gamma(T,e))=C_{\Gamma'}(T',e')\in  \mathcal{C}_0^+(\Gamma')$ and so $C_{\Gamma''}(T',e')\in  \mathcal{C}_0^+(\Gamma'')$ by Lemma \ref{thm-1}. Again according to \eqref{CC-eq-18}, we also get $\sigma''(e')=-\sigma(e)$. Similarly,
we have $\sigma''(e')=-\sigma(e)$  when $C_\Gamma(T,e)\in \mathcal{C}_0^-(\Gamma)$.

From the above discussions,  the isomorphism $\varphi: G \rightarrow G'$ maps  the edge $e\in E(\Gamma)$ to the edge $\varphi(e)\in E(\Gamma'')$  with opposite signs. It implies that
  $\varphi(-\Gamma)=\Gamma''\sim \Gamma'$, i.e.,  $-\Gamma\simeq\Gamma'$.
\end{proof}
 Now,  we present an equivalent condition for sign-symmetric signed graphs.

\begin{thm}\label{C-p-4}
The signed graph $\Gamma=(G, \sigma)$ is sign-symmetric  if and only if  $\Gamma$  has a weak automorphism if  and only if $\varphi(\mathcal{C}_1^+)=\mathcal{C}_1^-$ and $\varphi(\mathcal{C}_0^+)=\mathcal{C}_0^+$ for some $\varphi\in Aut(G)$.
\end{thm}
\begin{proof}
 $\Gamma$ has a weak automorphism if and only if $\varphi(\mathcal{C}_1^+)=\mathcal{C}_1^-$ and $\varphi(\mathcal{C}_0^+)=\mathcal{C}_0^+$ for some $\varphi\in Aut(G)$ by Lemma \ref{lemma-33}.

 We only need to prove that $\Gamma$ is sign-symmetric if and only if $\Gamma$ has a weak automorphism.
By setting $\Gamma'=\Gamma$ in Lemma \ref{conver-lem-1}, we immediately obtain  the sufficiency by Lemma \ref{lemma-33}.
Now, suppose that $\Gamma$ is sign-symmetric. Then there exist a switching $\Gamma'=(G, \sigma')\in [-\Gamma]$ and an isomorphism $f$  from $\Gamma$ to $\Gamma'$. Thus, $f(C_{\Gamma}(T,e))=C_{\Gamma'}(T',e')$, where $e'=f(e)$, $T$ and $T'=f(T)$ are two spanning trees of $G$.  Because $\Gamma$ and  $\Gamma'$ have the same underlying graph $G$, $f$ induces an automorphism $\varphi$ of $G$ that is just $f$ regardless of the signs of the edges of $\Gamma$ and $\Gamma'$. So, we have   $\varphi(C_{G}(T,e))=C_{G}(T',e')$.  If $C_{\Gamma}(T,e)\in \mathcal{B}_1^+(\Gamma, T)$, then $C_{\Gamma'}(T',e')\in \mathcal{B}_1^+(\Gamma', T')$  because $f$ is an isomorphism. Thus, the  cycle
  $C_{-\Gamma'}(T',e')$ of  $-\Gamma'$ is  negative  because of $C_{-\Gamma'}(T',e')= -C_{\Gamma'}(T',e')$. Since $\Gamma\in[-\Gamma']$ and $C_{-\Gamma'}(T',e')$ and $C_{\Gamma}(T',e')$ have the common underlying graph $C_{G}(T',e')$, so $C_{\Gamma}(T',e')$ is negative by Lemma \ref{thm-1}. It follows that   $C_{\Gamma}(T',e')\in\mathcal{B}_1^-(\Gamma, T') $ and so $\varphi(\mathcal{B}_1^+(\Gamma, T))=\mathcal{B}_1^-(\Gamma, T')$. Analogously,
if   $C_{\Gamma}(T,e)\in \mathcal{B}_0^+(\Gamma, T)$,
then  the cycle $C_{-\Gamma'}(T',e')$ of $-\Gamma'$ is positive.
Note that $C_{\Gamma}(T',e')$ and $C_{-\Gamma'}(T',e')$  have the same sign by Lemma \ref{thm-1}, we have $C_{\Gamma}(T',e')\in \mathcal{B}_0^+(\Gamma, T')$ and so
 $\varphi(\mathcal{B}_0^+(\Gamma, T))=\mathcal{B}_0^+(\Gamma, T')$. It follows that $\varphi$ is a weak   automorphism of $\Gamma$ by Theorem \ref{cor-exchan-1}.

We complete the proof.
\end{proof}

  Lemma \ref{thm-1} and Theorem \ref{C-p-4} imply the following result.
\begin{cor}\label{c-1}
$\Gamma=(G,\sigma)$ has a weak automorphism $\varphi$ if and only if $\varphi$ is a weak automorphism of $\Gamma'\in[\Gamma]$.
\end{cor}
According to Theorem \ref{C-p-4},  $\Gamma=(G,\sigma)$ is sign-symmetric if and only if  there exists a $\varphi\in Aut(G)$ such that   $C$ and $\varphi(C)$ have opposite signs for each odd cycle $C$ of $\Gamma$, and $C$ and $\varphi(C)$ have the same sign for each even cycle $C$ of $\Gamma$. So, we have
 \begin{lem}\label{wa-lem-1}
 Let $\Gamma=(G,\sigma)$ be a sign-symmetric signed graph with a weakly automorphism $\varphi$. Then $\varphi$ is of even order, i.e., $\varphi^{2r}=\iota$ for some positive integer $r$.
 \end{lem}
By Lemma \ref{wa-lem-1} we  give a feature for the underlying graph of a sign-symmetric graph.
\begin{cor}\label{C-p-4-1}
If  $\Gamma=(G,\sigma)$ is sign-symmetric, then $G$ has a transposition automorphism. In other words,   $\Gamma$ is not  sign-symmetric if  $G$ has no  transposition automorphism.
\end{cor}

Let  $G$ be a connected non-bipartite graph and   $\mathbb{S}(G)$  the set of all sign-symmetric signed graphs    on $G$.
A weak automorphism $\varphi$ of a signed graph $\Gamma=(G, \sigma)$ is called \emph{principal} if $\varphi$  has the smallest order among all weak automorphisms of $\Gamma$ (abbreviated as PWA-$\varphi$).
 For an integer $r\ge 1$, let
$$\mathbb{S}_r(G)=\{\Gamma\in\mathbb{S}(G)\mid \mbox{$\Gamma$ has    a PWA-$\varphi$  of order $2r$ } \}.$$
Clearly, we have $\mathbb{S}_i(G)\cap \mathbb{S}_j(G)=\emptyset$ for $i\not=j$ by the definition. By  Theorem \ref{C-p-4} and Lemma \ref{wa-lem-1}, we have
\begin{prop}\label{pro-1}
$\mathbb{S}(G)=\mathbb{S}_1(G)\cup \mathbb{S}_2(G)\cup\mathbb{S}_3(G)\cup\cdots$.
 \end{prop}
 \begin{figure}[h]
  \centering
  \includegraphics[width=13.5cm]{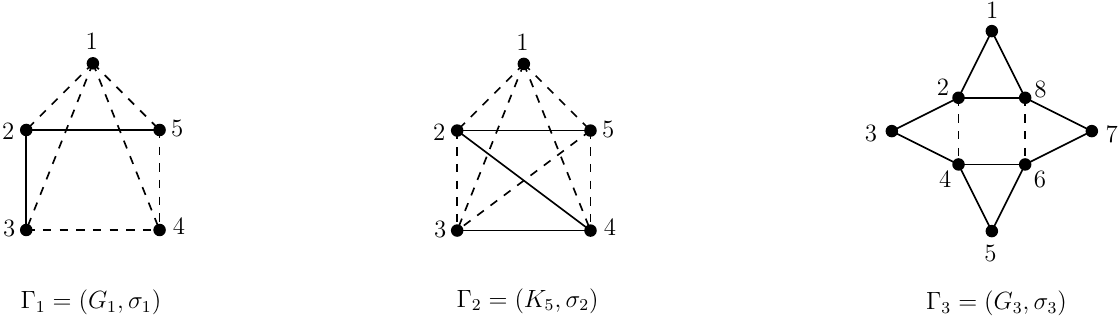}\\
  \caption{\small $\Gamma_1\in \mathbb{S}_1(G_1)$, $\Gamma_2\in \mathbb{S}_2(K_5),$ and $ \Gamma_3\in \mathbb{S}_1(G_3)$}\label{order4}
\end{figure}

At last of this section, we give Example \ref{ex3} to illustrate  Theorem \ref{C-p-4} and Proposition \ref{pro-1}.
 \begin{example}\label{ex3}
By Theorem \ref{cor-exchan-1} one can verify that $\Gamma_1=(G_1,\sigma_1)$ (see Fig.$\!$ \ref{order4}) has a weak automorphism $\varphi_1=(1)(2\ 4)$ $(3\ 5)$. By Theorem \ref{C-p-4}, $\Gamma_1$ is  sign-symmetric and thus $\Gamma_1\in \mathbb{S}_1(G)$.  $\Gamma_2$ and $\Gamma_3$ in Fig.$\!$ \ref{order4} are two  sign-symmetric graphs because
$\varphi_2=(1)(2\ 3\ 4\ 5)$ and $\varphi_3=(1\ 3\ 5\ 7)(2\ 4\ 6\ 8)$ are  respectively  weak  automorphisms  of  $\Gamma_2$ and  $\Gamma_3$. On the other aspect, $\varphi_3'=(1\ 3)(4\ 8)(5\ 7)(2)(6)$ is also a  weak  automorphism  on $\Gamma_3$. Hence, $\varphi_3'$ is a PWA-$\varphi$ of $\Gamma_3$ but not $\varphi_3$. Thus $\Gamma_3\in \mathbb{S}_1(G)$. At last, we claim that $\Gamma_2\in \mathbb{S}_2(G)$. Since otherwise, without loss of generality, $\Gamma_2$ has a  weak automorphism  $f=(1)(2 \ 3)(4\ 5)$. Then $C=123$ is a triangle  and $f(C)=C$. It  leads to a contradiction  whether $\sigma(C)=1$ or $\sigma(C)=-1$.
\end{example}
 In the next section, we focus on determining  the signed graphs in $\mathbb{S}_1(G)$.

\section{Complete characterization of the signed graphs  in $\mathbb{S}_1(G)$}

In this section, we always assume that   $\Gamma=(G,\sigma)\in \mathbb{S}_1(G)$ with a PWA-$\varphi$.
Then
$V(G)$ has a vertex partition
$V(G)=F\cup U_1\cup U_2$
such that $\varphi(U_1)=U_2$, $\varphi(U_2)=U_1$, and $F=\{v\in V(G)\mid \varphi(v)=v\}$. It follows that  $G[U_2]=\varphi(G[U_1])$ and the permutation \begin{equation}\label{CC-eq-21}\varphi=\prod_{u\in U_1}(u\ \varphi(u))\end{equation}
is a transposition with  \emph{fixed vertices} in $F$. For any $u\in U_1$,
 $u$ and $\varphi(u)$ are called a pair of \emph{symmetric vertices}.

 \begin{figure}[h]
  \centering
  \includegraphics[width=12cm]{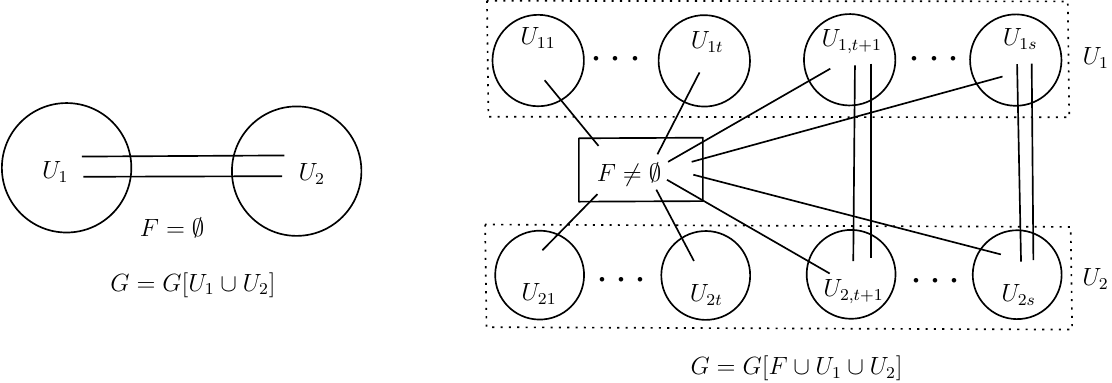}\\
  \caption{\small   $\Gamma=(G,\sigma)\in \mathbb{S}_1(G)$ has two types:  $F=\emptyset$ and $F\not=\emptyset$.}\label{figure47}
\end{figure}
Note that $G$ is  connected. If $F=\emptyset$, then $G$ consists of two isomorphic parts $G[U_1]$ and $ G[U_2]$, and  some   edges between them, which is shown in the left of Fig.$\!$ \ref{figure47}.
If $F\not=\emptyset$, then, by deleting the vertex set $F$ from $G$, we obtain some components in $G[U_1\cup U_2]$ depicted in the right of Fig.$\!$ \ref{figure47}.
\begin{clm}\label{clm-4-0}
Let $\Gamma=(G,\sigma)\in \mathbb{S}_1(G)$ be a signed graph with a PWA-$\varphi$.  \\
(i) There exists a vertex partition $U_i=\cup_{j=1}^s U_{ij}$ ($i=1,2$) such that $\varphi(G[U_{2j})=G[U_{1j}]$ is connected for  $j=1, \ldots, s$.\\
(ii) $G[U_1\cup U_2]$ is a union of components: $G[U_1\cup U_2]= (\cup_{j=1}^tG[U_{1j}])\cup (\cup_{j=1}^tG[U_{2j}])\cup(\cup_{j=t+1}^sG[U_{1j}\cup U_{2j}])$, where  $G_j=G[U_{1j}\cup U_{2j}]$ is called a \emph{symmetric factor} of $\Gamma$  for $j=1, \ldots, t$, and a \emph{reflexive factor} of $\Gamma$ for $j=t+1, \ldots, s$ (see  Fig.\ref{figure47}). Reflexive and symmetric factors are collectively referred to as $\varphi$-factors of $\Gamma$.
\end{clm}
From Claim \ref{clm-4-0}, a symmetric factor  contains two isomorphic components, while reflexive factor itself is a component in $G[U_1\cup U_2]$. Let $G_j=G[U_{1j}\cup U_{2j}]$ be a reflexive $\varphi$-factor. Denote by $E_{G_j}=E_{G_j}(U_{1j}, U_{2j})$ the set of edges between $U_{1j}$ and $ U_{2j}$ in $G_j$.
From (\ref{CC-eq-21}), the edge $e\in E_{G_j}$ can be presented by $e=u\varphi(v)$ for some $u,v\in U_{1j}$. $e$ is called a symmetric edge if $u\not=v$, and a reflexive edge if $u=v$. Denote by $E_{G_j}^s$ and $E_{G_j}^r$  the  symmetric edge set and reflexive edge set, respectively. Obviously, $u\varphi(v)\in E_{G_j}^s$ if and only if $\varphi(u)v\in E_{G_j}^s$. We call $u\varphi(v), v\varphi(u)\in E_{G_j}^s$  a pair of symmetric edges. An edge $e\in E_{G_j}$ is  a reflexive edge if and only if $\varphi(e)=e$.

For simplicity's sake, we always use $uv\in E_{G}(U, V)$ to denote the edge $uv$ satisfying $u\in U$ and $v\in V$.
Let  $\hat{G}_j=G[F\cup U_{1j}\cup U_{2j}]$. Clearly,  $E_{\hat{G}_j}(F, U_{1j})\not=\emptyset$ and  $E_{\hat{G}_j}(F, U_{2j})\not=\emptyset$ if $F\not=\emptyset$, and    $e=\omega u\in E_{\hat{G}_j}(F, U_{1j})$ if and only if  $e'=\varphi(e)=\omega\varphi(u)\in E_{\hat{G}_j}(F, U_{2j})$. Such  $e$ and $e'$ are called a pair of \emph{rotated edges}. Moreover, let $\Gamma_j=(G_j,\sigma)$ and $\hat{\Gamma}_j=  (\hat{G}_j,\sigma)$. In what follows, we will use the above symbols without declaration.
 \begin{clm}\label{clm-4-1}
Let $\Gamma=(G,\sigma)\in \mathbb{S}_1(G)$ be a signed graph with a PWA-$\varphi$. For $j=1, \ldots, s$, we have\\
 (i) $E_{G_j}=E_{G_j}^s\cup E_{G_j}^r$ if $G_j$ is a reflexive $\varphi$-factor;\\
  (ii)  $E_{\hat{G}_j}(F,U_{1j}\cup U_{2j})=E_{\hat{G}_j}(F, U_{1j})\cup E_{\hat{G}_j}(F, U_{2j})$ and $\varphi(E_{\hat{G}_j}(F, U_{1j}))=E_{\hat{G}_j}(F, U_{2j})$ if $G_j$ is a  symmetric $\varphi$-factor;\\
  (iii) the restrictions of  $\varphi$  on $\Gamma_j$ and $\hat{\Gamma}_j$ both are  weak automorphisms of order 2.\\
 (iv) $G[F]=(\Omega_1,\Omega_2)$ is a bipartite graph.
\end{clm}
It is worth mentioning that $G[U_1\cup U_2]$ itself is  a reflexive $\varphi$-factor  if it is connected. By   Lemmas \ref{thm-1},  \ref{FF-lem-2} and \ref{conver-lem-1}, we have
\begin{lem}\label{lem-4-2}
$\varphi(-\Gamma_j[U_{1j}])\sim\Gamma_j[U_{2j}]$ for $j=1, \ldots, s$.
\end{lem}
\begin{proof}
Let $\Gamma_{1j}=\Gamma[U_{1j}]=(G[U_{1j}],\sigma)$ and $\Gamma_{2j}=\Gamma[U_{2j}]=(G[U_{2j}],\sigma)$ be two subgraphs of $\Gamma_j$. By Claim \ref{clm-4-0} (i), $\varphi:G[U_{1j}]\longrightarrow G[U_{2j}]$ is an isomorphism   such that  $\varphi(G[U_{1j}])=G[U_{2j}]$. By Claim \ref{clm-4-1} (iii),  $\varphi$ is a weak automorphism of $\Gamma_j$. By Theorem \ref{C-p-4} we have $\varphi(\mathcal{C}_1^+(\Gamma_{1j}))=\mathcal{C}_1^-(\Gamma_{2j})$ and $\varphi(\mathcal{C}_0^+(\Gamma_{1j}))=\mathcal{C}_0^+(\Gamma_{2j})$. Thus, we have $\varphi(-\Gamma_j[U_{1j}])\sim\Gamma_j[U_{2j}]$ by Lemma \ref{conver-lem-1}.
\end{proof}
According to Claim \ref{clm-4-1} (iv), the adjacency matrix  $A(\hat{\Gamma}_j)$ of $\hat{\Gamma}_j=(\hat{G}_j,\sigma)$ can be expressed by
$$A(\hat{\Gamma}_j)=\left(\begin{array}{cc|cc}
 0 & A_{\Omega_1, \Omega_2} & A_{\Omega_1, U_{1j}} & A_{\Omega_1, U_{2j}}  \\
 A_{\Omega_1, \Omega_2}^\top & 0 & A_{\Omega_2, U_{1j}} & A_{\Omega_2, U_{2j}} \\ \hline
 A_{\Omega_1, U_{1j}}^\top & A_{\Omega_2, U_{1j}}^\top & A(\Gamma[U_{1j}]) & A_{U_{1j}, U_{2j}}  \\
  A_{\Omega_1, U_{2j}}^\top &  A_{\Omega_2, U_{2j}}^\top & A_{U_{1j}, U_{2j}}^\top & A(\Gamma[U_{2j}])
  \end{array} \right).$$
By Lemma \ref{lem-4-2}, we have $-\Gamma_j[U_{1j}] \sim \varphi(\Gamma_j[U_{2j}])$ due to $\varphi^2=\iota$. Thus, if  the vertices of $U_{2j}$ are labelled as $U_{2j}=\{\varphi(u)\mid u\in U_{1j}\}$, then there is a $(1,-1)$-diagonal matrix $D_j$ such that $-A(\Gamma_{j}[U_{1j}])=D_jA(\Gamma_{j}[U_{2j}])D_j$.
Setting $\tilde{D}_j=I_{|F|}\oplus(I_{|U_{1j}|}\oplus D_j)$ where $I_{|F|}$ is the identity matrix of order $|F|$, we have
$$\tilde{D}_jA(\hat{\Gamma}_j)\tilde{D}_j=
\left(\begin{array}{cc|cc}
 0 & A_{\Omega_1, \Omega_2} & A_{\Omega_1, U_{1j}} & A_{\Omega_1, U_{2j}} D_j  \\
 A_{\Omega_1, \Omega_2}^\top & 0 & A_{\Omega_2, U_{1j}} & A_{\Omega_2, U_{2j}} D_j  \\ \hline
 A_{\Omega_1, U_{1j}}^\top & A_{\Omega_2, U_{1j}}^\top & A(\Gamma[U_{1j}]) & A_{U_{1j}, U_{2j}} D_j  \\
   D_jA_{\Omega_1, U_{2j}}^\top &  D_j A_{\Omega_2, U_{2j}}^\top & D_j A_{U_{1j}, U_{2j}}^\top & -A(\Gamma[U_{1j}])
  \end{array}\right). $$
In general, for  $\Gamma=(G,\sigma)\in \mathbb{S}_1(G)$, let  $\tilde{D}=I_{|F|}\oplus (I_{|U_{11}|}\oplus\tilde{D}_1)\oplus\cdots\oplus (I_{U_{|1s|}}$ $\oplus\tilde{D}_s)$.  By Lemma \ref{lem-4-2}, we get a signed graph $\tilde{\Gamma}=(G,\tilde{\sigma})\simeq \Gamma$ via $A(\tilde{\Gamma})=\tilde{D}P A(\Gamma)P^\top\tilde{D}$,  where $P$ is a permutation matrix on $V(G)$ that relabels the vertices  $U_{2j}=\{\varphi(u)\mid u\in U_{1j}\}$ for  $j=1, \ldots, s$ and fixes the other vertices.
Thus, $\tilde{\Gamma}\in \mathbb{S}_1(G)$ by Lemma \ref{FF-lem-2}. It means that, up to switching isomorphism,     $\hat{\Gamma}_j=(G_j,\sigma)$  has the  adjacency matrix:
\begin{equation}\label{CC-eq-15}
A(\hat{\Gamma}_j)=\left(\begin{array}{cc|cc}
 0 & A_{\Omega_1, \Omega_2} & A_{\Omega_1, U_{1j}} & A_{\Omega_1, U_{2j}}   \\
 A_{\Omega_1, \Omega_2}^\top & 0 & A_{\Omega_2, U_{1j}} & A_{\Omega_2, U_{2j}}  \\ \hline
 A_{\Omega_1, U_{1j}}^\top & A_{\Omega_2, U_{1j}}^\top & A(\Gamma[U_{1j}]) & A_{U_{1j}, U_{2j}}   \\
   A_{\Omega_1, U_{2j}}^\top &   A_{\Omega_2, U_{2j}}^\top &  A_{U_{1j}, U_{2j}}^\top & -A(\Gamma[U_{1j}])
  \end{array} \right), \mbox{ where $\varphi(U_{1j})=U_{2j}$ }.
   \end{equation}
Therefore,   in the following, we will always assume that $\hat{\Gamma}_j=(\hat{G}_j,\sigma)$ has the adjacency matrix (\ref{CC-eq-15}). Notice that the edge signs of  $\Gamma[U_{1j}]$ and $\Gamma[\Omega_1\cup \Omega_2]$ may be arbitrary, so $\hat{\Gamma}_j$ is  determined by $A_{U_{1j}, U_{2j}}$ and $A_{\Omega_{i'}, U_{ij}}$ for $i', i\in\{1,2\}$.

  Let  $G_j=G[U_{1j}\cup U_{2j}]$ be a reflexive $\varphi$-factor.  We  respectively define two signed graphs $\Gamma_j^a=(G_j,\sigma^a)$ and $\Gamma_j^b=(G_j,\sigma^b)$ by their adjacency matrices
\begin{equation}\label{Gamma-j-eq-1}A(\Gamma_j^a)=\bordermatrix{
   &U_{1j} & U_{2j} \cr
U_{1j}& B&X \cr
 U_{2j}&X&-B  \cr}\ \ \mbox{ and }\ \ A(\Gamma_j^b)=\bordermatrix{
   &U_{1j} & U_{2j} \cr
U_{1j}& B&Y \cr
 U_{2j}&-Y&-B  \cr},
\end{equation}
where $B$ is a symmetric $(0,1,-1)$-matrix with a zero diagonal, $X=X^\top$ is a symmetric $(0,1,-1)$-matrix, and $Y=-Y^\top$ is an antisymmetric  $(0,1,-1)$-matrix. By setting  $P_1=\left(\begin{array}{cc}
0&I\\
-I&0
\end{array}\right)$ and $P_2=\left(\begin{array}{cc}
0&I\\
I&0
\end{array}\right)$, we can simply verify that
$P_1A(\Gamma^a_j)P_1^\top=-A(\Gamma^a_j)$ and
$P_2A(\Gamma^b_j)P_2^\top=-A(\Gamma^b_j)$. The following result is obtained.

\begin{lem}\label{lem-4-6}
$\Gamma_j^a$ and $\Gamma^b_j$ are sign-symmetric.
\end{lem}

We use $P_{u, v}$ to denote a path  between the vertices $u$ and $v$.
\begin{lem}\label{pro-4-2}
Let  $G_j=G[U_{1j}\cup U_{2j}]$ be a reflexive $\varphi$-factor. For two edges $e_1=u_1v_1,e_2=u_2 v_2\in E_{G_j}(U_{1j}, U_{2j})$, we have
$\prod_{i=1}^2\sigma(e_i)\sigma(\varphi(e_i))=1.$
\end{lem}
\begin{proof}
 $G[U_{1j}]$ and $G[U_{2j}]$ are connected by Claim \ref{clm-4-0} (i). Then there exist two paths $P_{u_2, u_1}\subseteq G[U_{1j}]$ and $P_{v_1, v_2}\subseteq G[U_{2j}]$ such that $C=P_{u_2, u_1}+u_1v_1+P_{v_1,v_2}+v_2 u_2$ is a cycle of $\Gamma_j$ containing $e_1$ and $e_2$.  Since $\Gamma[U_{2j}]=-\Gamma[U_{1j}]$ by (\ref{CC-eq-15}),  we  have $\sigma(\varphi(e))=-\sigma(e)$ for all $e\in E(\Gamma[U_{ij}])$, $i=1, 2$.  So, we have
\begin{equation}\label{CC-eq-20}
\begin{array}{ll}
\sigma(\varphi(C))&=\sigma(\varphi(e_1))
\sigma(\varphi(e_2))\cdot \prod_{e\in E(C)\setminus\{e_1,e_2\}}\sigma(\varphi(e))\\[0.2cm]
&=\sigma(\varphi(e_1))
\sigma(\varphi(e_2))\cdot(-1)^{|C|-2}\cdot\sigma(C)\sigma(e_1)\sigma(e_2)\\[0.2cm]
&=(-1)^{|C|-2}\sigma(C)\prod_{i=1}^2\sigma(e_i)\sigma(\varphi(e_i)).
\end{array}
\end{equation}
From Theorem \ref{C-p-4} and  (\ref{CC-eq-20}), the result holds  regardless of whether $C$ is an even cycle or an odd cycle.
\end{proof}

\begin{lem}\label{pro-4-1}
Let $\Gamma=(G,\sigma)\in \mathbb{S}_1(G)$ and $G_j=G[U_{1j}\cup U_{2j}]$ is a reflexive $\varphi$-factor.  Then $\Gamma_j\simeq\Gamma_j^a$ or $\Gamma_j^b$, where $\Gamma_j^a$ and $\Gamma_j^b$ are defined by (\ref{Gamma-j-eq-1}).
\end{lem}
\begin{proof}
Since $\Gamma=(G, \sigma)\in\mathbb{S}_1(G)$, from (\ref{CC-eq-15}),  $A(\Gamma_j)$ has the form
$$A(\Gamma_j)=\left(\begin{array}{cc}
A(\Gamma[U_{1j}])&A_{U_{1j},U_{2j}}\\
A_{U_{2j},U_{1j}} &-A(\Gamma[U_{1j}])
\end{array}\right).
$$
We  take $B$ as $A(\Gamma[U_{1j}])$. It remains to verify that $A_{U_{1j},U_{2j}}=(a_{u\varphi(u)})$ is either a symmetric or antisymmetric $(0,1,-1)$-matrix. First of all, if $E_{G_j}=E_{G_j}(U_{1j}, U_{2j})$ contains only one edge $e$, then $e$ is a   reflexive edge dose not included in any cycle of $\Gamma_j$. Obviously, $A_{U_{1j},U_{2j}}$ is a symmetric $(0,1,-1)$-matrix having only one $1$ or $-1$ at its diagonal. Thus we have $\Gamma_j\simeq \Gamma_j^a$ in this case.

Let $e_1, e_2\in E_{G_j}$ with $e_1\not=e_2$.
For $i=1,2$, $e_i$ and $e_i'=\varphi(e_i)$ form a pair of symmetric edges if $e_i\not=\varphi(e_i)$,  otherwise,  $e_i=\varphi(e_i)$ is a reflexive edge.
We get $E_{G_j}=E_{G_j}^s\cup E_{G_j}^r$ by
Claim \ref{clm-4-1} (i). From Lemma \ref{pro-4-2}, we have
\begin{equation}\label{sr-eq-1}1 =\prod_{i=1}^2\sigma(e_i)\sigma(\varphi(e_i))=\left\{\begin{array}{ll}
\sigma(e_1)^2\sigma(e_2)^2 & \mbox{ if $e_2=\varphi(e_1)$,}\\
\sigma(e_1)\sigma(e_1')\sigma(e_2)\sigma(e_2') & \mbox{ if $e_2\not=\varphi(e_1)$ and $e_1, e_2\in E_{G_j}^s$,} \\
\sigma(e_1)\sigma(e_1')\sigma(e_2)^2& \mbox{ if $e_1\in E_{G_j}^s$ and  $e_2\in E_{G_j}^r$,}\\
\sigma(e_1)^2\sigma(e_2)^2 & \mbox{ if $e_1,e_2\in E_{G_j}^r$.}\\
\end{array}\right.
\end{equation}
If $E_{G_j}^s=\emptyset$, then $E_{G_j}=E_{G_j}^r$ and $e_1=u_1\varphi(u_1),e_2=u_2\varphi(u_2)\in E_{G_j}^r$. From (\ref{sr-eq-1}) we have
$\sigma(e_1)^2\sigma(e_2)^2=1$ and so $a_{u_i\varphi(u_i)}=\pm1$. Thus, $A_{U_{1j},U_{2 j}}$ is a diagonal matrix with $a_{u_i\varphi(u_i)}=\pm1$ corresponding  edge $e_i\in E_{G_j}^r$. If $E_{G_j}^r=\emptyset$, then $E_{G_j}=E_{G_j}^s$ and $e_1=u_1\varphi(v_1),e_2=u_2\varphi(v_2)\in E_{G_j}^s$. From (\ref{sr-eq-1}) we have $\sigma(e_1)\sigma(e_1')=\sigma(e_2)\sigma(e_2')$. It implies that $A_{U_{1j},U_{2j}}$ is symmetric if $\sigma(e_1)=\sigma(e_1')$,  or antisymmetric if $\sigma(e_1)=-\sigma(e_1')$. If $E_{G_j}^s\not=\emptyset$ and $E_{G_j}^r\not=\emptyset$, by taking $e_1=u_1\varphi(v_1)\in E_{G_j}^s$ and $e_2=u_2\varphi(u_2)\in E_{G_j}^r$,  we have $\sigma(e_1)\sigma(e_1')\sigma(e_2)^2=1$ from (\ref{sr-eq-1}) and so $\sigma(e_1)\sigma(e_1')=1$. It implies that  $A_{U_{1j},U_{2j}}$ is symmetric with diagonal entry $a_{u_2\varphi(u_2)}=\pm1$ corresponding edge $e_2\in E_{G_j}^r$.

 Summarising above analysis, we have
$$\Gamma_j\simeq\left\{\begin{array}{ll}
\Gamma_j^a, & \mbox{if $G_j$ contains  reflexive edges},\\
\Gamma_j^a~~ \mbox{or}~~ \Gamma_j^b,  & \mbox{if $G_j$ only contains   symmetric edges.}\\
\end{array}\right.$$
\end{proof}

\begin{re}
  It is worth mentioning that   Lemmas \ref{pro-4-1} and \ref{lem-4-6} determine the local structure of  signed graphs in $\mathbb{S}_1(G)$.
\end{re}

Let $\mathbb{NFS}_1(G)=\{\Gamma\in \mathbb{S}_1(G)\mid F=\emptyset\}$ and $\mathbb{FS}_1(G)=\{\Gamma\in \mathbb{S}_1(G)\mid F\not=\emptyset\}$.
First,  we  characterize the signed graphs in $\mathbb{NFS}_1(G)$. Obviously, $G[U_1\cup U_2]$ is  a reflexive $\varphi$-factor  in this situation. By Lemmas \ref{lem-4-6} and \ref{pro-4-1}, we immediately get the following result.
\begin{thm}[Construction Theorem-I]\label{thm-main-1}
$\Gamma=(G,\sigma)\in \mathbb{NFS}_1(G)$  if and only if $\Gamma\simeq\Gamma^a$ or $\Gamma^b$.
\end{thm}

 Theorem \ref{thm-main-1} implies Theorem \ref{thm-10}, which was obtained in \cite{Akbaria}, and Theorem \ref{thm-12}, which was obtained in \cite{Ghorbani}. In fact, the sign-symmetric graphs determined in  Theorems \ref{thm-10} and  \ref{thm-12} are identical to $\Gamma^a$.  $\Gamma^b$ is new,  and moreover,  the result of Theorem \ref{thm-main-1} is both sufficient and necessary.

\begin{thm}[$\!$\cite{Akbaria}]\label{thm-10}
Let $n$ be an even positive integer and $V_1$ and $V_2$ be two disjoint sets of size
$\frac{n}{2}$. Let $G$ be an arbitrary graph with the vertex set $V_1$. Construct the complement $\overline{G}$ of $G$ with the vertex set $V_2$. Assume that $\Gamma = (K_n, \sigma)$ is a signed complete graph in which $E(G) \cup E(\overline{G})$ is the set of negative edges. Then  $\Gamma$  is sign-symmetric.
\end{thm}

\begin{thm}[$\!$\cite{Ghorbani}]\label{thm-12}
The signed graph $\Gamma$ with the adjacency matrices
$A(\Gamma)=\left(\begin{array}{cc}
B&C\\
C&-B\\
\end{array}\right)$ is sign-symmetric, where $B, C $ are symmetric matrices with  entries from $\{0, 1, -1\}$ and $B$ has a zero diagonal.
\end{thm}

Next, we determine  the signed graphs in $\mathbb{FS}_1(G)$.    We  present some lemmas for preparation.
\begin{lem}\label{pro-4-0}
Let $e_1=\omega_1u_1\in E_{\hat{G}_j}(F,U_{1j}), e_2=\omega_2v_1\in E_{\hat{G}_j}(F, U_{2j})$ and  $e_3=u_2 v_2\in E_{G_j}(U_{1j},U_{2j})$, where $\omega_1$ may be equal to $\omega_2$. If $G[F]=(\Omega_1,\Omega_2)$ is connected,  then\\
(i) $\prod_{i=1}^{3}\sigma(e_i)\sigma(\varphi(e_i))=1$ if  $\omega_1\in \Omega_1$, $\omega_2\in \Omega_2$, or  $\omega_1\in \Omega_2$, $\omega_2\in \Omega_1$;\\
(ii) $\prod_{i=1}^{3}\sigma(e_i)\sigma(\varphi(e_i))=-1$ if $\omega_1, \omega_2\in \Omega_1$,  or $\omega_1, \omega_2\in \Omega_2$.
\end{lem}
\begin{proof}
By assumption, there exist three paths $P_{\omega_1,\omega_2}\subseteq G[F]$, $P_{u_1, u_2}\subseteq G[U_{1j}]$ and $P_{v_1, v_2}\subseteq G[U_{2j}]$ such that
\begin{equation}\label{CC-eq-1}
C=P_{u_1, u_2}+\omega_1u_1+P_{\omega_1,\omega_2}+\omega_2v_1+P_{v_1,v_2}+u_2 v_2
\end{equation}
 is a cycle of $\hat{\Gamma}_j$. Since $\hat{\Gamma}_j[U_{2j}]=-\hat{\Gamma}_j[U_{1j}]$
from (\ref{CC-eq-15}), we have $\sigma(\varphi(e))=-\sigma(e)$ if  $e$ is an edge of   $P_{u_1, u_2}$ or $P_{v_1,v_2}$. Additionally,  $\varphi$ fixes each edge of $P_{\omega_1,\omega_2}$. Therefore, we have
\begin{equation}\label{qq-0}\begin{array}{ll}
\sigma(\varphi(C))&=\sigma(\varphi(e_1))
\sigma(\varphi(e_2))\sigma(\varphi(e_3))\cdot \prod_{e\in E(C)\setminus\{e_1,e_2,e_3\}}\sigma(\varphi(e))\\[0.2cm]
&=(-1)^{|C|-|P_{\omega_1,\omega_2}|-3}\cdot\sigma(C)\cdot\prod_{i=1}^{3}
\sigma(e_i)\sigma(\varphi(e_i)).
\end{array}
\end{equation}
Note that  $|P_{\omega_1,\omega_2}|$ is even if $\omega_1,\omega_2\in\Omega_1$ or $\omega_1,\omega_2\in \Omega_2$, and  $|P_{\omega_1,\omega_2}|$ is odd otherwise. The lemma follows by (\ref{qq-0}) and Theorem \ref{C-p-4}.
\end{proof}

\begin{lem}\label{lemma-20}
Suppose that $e=\omega_1 u_1\in E_{\hat{G}_j}(F, U_{1j})$.\\
(i) If  $ e_1=\omega_1 u_2\in E_{\hat{G}_j}(F, U_{1j})$, then $\sigma(e)\sigma(\varphi(e))=\sigma(e_1)\sigma(\varphi(e_1))$.\\
(ii) If $G[F]=(\Omega_1,\Omega_2)$ is connected and $e_2=\omega_2u_1\in E_{\hat{G}_j}(F,U_{1j})$, then $\sigma(e)\sigma(\varphi(e))\sigma(e_2)$ $\sigma(\varphi(e_2))=1$ if $\omega_1$ and $\omega_2$ in the same partition, and $\sigma(e)\sigma(\varphi(e))\sigma(e_2)\sigma(\varphi(e_2))=-1$ otherwise.
\end{lem}
\begin{proof}
$G[U_{1j}]$ and $G[U_{2j}]$ are connected by Claim \ref{clm-4-0} (i).
Let  $P_{u_1,u_2}\subseteq G_j[U_{1j}]$ be a path between $u_1$ and $u_2$. Then $C=e+e_1+P_{u_1,u_2}$ and $\varphi(C)=\omega_1 \varphi(u_1)+\omega_2\varphi(u_2)+\varphi(P_{u_1,u_2})$ are two  cycles of $\hat{\Gamma}_j$. Since $\sigma(\varphi(e))=-\sigma(e)$ for each edge $e\in P_{u_1,u_2}$, we have
\begin{equation}\label{qq-01}\begin{array}{ll}
\sigma(\varphi(C))&=\sigma(\varphi(e))
\sigma(\varphi(e_1))\cdot \prod_{e\in E(C)\setminus\{e,e_1\}}\sigma(\varphi(e))\\[0.2cm]
&=(-1)^{|P_{u_1,u_2}|}\cdot\sigma(C)\cdot
\sigma(e)\sigma(\varphi(e))\sigma(e_1)\sigma(\varphi(e_1)).
\end{array}
\end{equation}
It follows (i) from Theorem \ref{C-p-4}.

Since $G[F]$ is connected, there exists a path $P_{\omega_1,\omega_2}\subseteq G[F]$. Clearly, $C=P_{\omega_1,\omega_2}+e+e_2$ is a cycle of $\hat{\Gamma}_j$ and  $\varphi(C)=P_{\omega_1,\omega_2}+\omega_1 \varphi(u_1)+\omega_2 \varphi(u_1)$.  We have $$
\sigma(\varphi(C))=\sigma(P_{\omega_1,\omega_2})
\sigma(\varphi(e))\sigma(\varphi(e_2))\\
=\sigma(C)\sigma(e)\sigma(\varphi(e))\sigma(e_2)\sigma(\varphi(e_2)).
$$
It follows (ii)  from Theorem \ref{C-p-4}.
\end{proof}
Let $\hat{\Gamma}_j^a$ and $\hat{\Gamma}_j^b$ be two signed graphs on $\hat{G}_j$ defined    by
\begin{small}
\begin{equation}\label{FF-eq-1}
A(\hat{\Gamma}_j^a)=\bordermatrix{
   &\Omega_1 &  \Omega_2 &U_{1j} & U_{2j} \cr
\Omega_1 &0 & A & D_1 & D_1 \cr
 \Omega_2& A^\top & 0 & D_2 & -D_2  \cr
      U_{1j}&      D_1^\top & D_2^\top & B & Y \cr
      U_{2j}  &                   D_1^\top & -D_2^\top & -Y & -B \cr}\ \   \mbox{and} \ \   A(\hat{\Gamma}_j^b)=\bordermatrix{
   &\Omega_1 &  \Omega_2 &U_{1j} & U_{2j} \cr
\Omega_1 &0 & A & D_1 & -D_1 \cr
 \Omega_2& A^\top & 0 & D_2 &D_2  \cr
      U_{1j}&      D_1^\top & D_2^\top & B & Y \cr
      U_{2j}  &                   D_1^\top & -D_2^\top & -Y & -B \cr}
\end{equation}
\end{small}
where   $A$, $B$, $Y$, $D_1$ and $D_2$ are $(0, -1, 1)$-matrices,  $B^\top=B$ is of zero diagonal, and  $Y^\top=-Y$. By setting
$$P_1=\left(
         \begin{array}{cc|cc}
           -I & 0 & 0 & 0 \\
           0 & I & 0 & 0 \\ \hline
           0 & 0 & 0 & I \\
           0 & 0 & I & 0 \\
         \end{array}
       \right)\ \ \mbox{ and } \ \ P_2=\left(
         \begin{array}{cc|cc}
           -I & 0 & 0 & 0 \\
           0 & I & 0 & 0 \\
           \hline
           0 & 0 & 0 & -I \\
           0 & 0 & I & 0 \\
         \end{array}
       \right),$$
one can verify that $P_1A(\hat{\Gamma}_j^a)P_1^\top=-A(\hat{\Gamma}_j^a)$ and $P_2A(\hat{\Gamma}_j^b)P_2^\top=-A(\hat{\Gamma}_j^b)$. Thus,  the following result holds.
\begin{lem}\label{lem-4-61}
$\hat{\Gamma}_j^a$ and $\hat{\Gamma}^b_j$ are sign-symmetric.
\end{lem}
\begin{lem}\label{FF-lem-1}
 If $G[F]=(\Omega_1,\Omega_2)$ is connected, then $\hat{\Gamma}_j\simeq\hat{\Gamma}_j^a$ or $\hat{\Gamma}_j^b$.
\end{lem}
\begin{proof}
First of all,  assume that $A(\hat{\Gamma}_j)$ is expressed as (\ref{CC-eq-15}). Take $A(\hat{\Gamma}_j[U_{1j}])=-A(\hat{\Gamma}_j[U_{2j}])=B$. Since $G[F]=(\Omega_1,\Omega_2)$ is bipartite, we may take $A_{\Omega_1, \Omega_2}=A$ and thus $A_{\Omega_2, \Omega_1}=A^\top$.
If $G_j=G[U_{1j}, U_{2j}]$ is a symmetric $\varphi$-factor, then $A_{U_{1j},U_{2j}}=Y=0$ because  $G_j$ has two components $G_j[U_{1j}]$ and $G_j[U_{2j}]$. If $G_j=G[U_{1j}, U_{2j}]$ is a reflexive $\varphi$-factor, by taking the edges $e_1=\omega u_1\in E_{\hat{G}_j}(F, U_{1j})$, $e_2=\varphi(e_1)=\omega\varphi(u_1)$, and  $e_3=u_2 v_2\in E_{G_j}(U_{1j},U_{2j})$,
we have   $-1=\prod_{i=1}^{3}\sigma(e_i)\sigma(\varphi(e_i))$ by Lemma \ref{pro-4-0} (ii).
It implies that $\sigma(e_3)\sigma(\varphi(e_3))=-1$.  Hence, $e_3$ and $\varphi(e_3)$ must be a pair of symmetric edges with opposite signs and so $A_{U_{1j},U_{2j}}=Y$ is an antisymmetric matrix.

It remains to verify that \begin{equation}\label{U-eq-1}\left(\begin{array}{cc}A_{\Omega_1,U_{1j}}&A_{\Omega_1,U_{2j}}\\
A_{\Omega_2,U_{1j}}&A_{\Omega_2,U_{2j}}\end{array}\right)=\left(\begin{array}{cc}D_1&D_1\\
D_2&-D_2\end{array}\right)\ \ \mbox{ or  } \ \ \left(\begin{array}{cc}D_1&-D_1\\
D_2&D_2\end{array}\right).\end{equation}
First, we can let $A_{\Omega_1,U_{1j}}=(a_{\omega_1u})=D_1$ and  $A_{\Omega_2,U_{1j}}=(a_{\omega_2u})=D_2$. Next, we  verify that $A_{\Omega_1,U_{2j}} =(a_{\omega_1 \varphi(u)})=D_1$ and  $A_{\Omega_2,U_{2j}}=(a_{\omega_2 \varphi(u)})=-D_2$, or $A_{\Omega_1,U_{2j}} =-D_1$ and  $A_{\Omega_2,U_{2j}}=D_2$.

Note that $\varphi(E_{\hat{G}_j}(F, U_{1j}))=E_{\hat{G}_j}(F, U_{2j})$ from Claim \ref{clm-4-1} (ii),
for any edge $e=\omega_1 u\in E_{\hat{\Gamma}_j}(\Omega_1,U_{1j})$, we have $\sigma(e)=a_{\omega_1u}\in \{\pm a_{\omega_1\varphi(u)}\}$. If $\sigma(\omega_1u)=\sigma(\omega_1\varphi(u))$, i.e. $a_{\omega_1u}=a_{\omega_1\varphi(u)}$,  by taking the edge $e'=\omega_1 u'\in  E_{\hat{\Gamma}_j}(\Omega_1,U_{1j})$, from Lemma \ref{lemma-20} (ii) we get $1=\sigma(e)\sigma(\varphi(e))=\sigma(e')\sigma(\varphi(e'))$
and so $a_{\omega_1u'}=a_{\omega_1\varphi(u')}$. By taking the edge $e''=\omega_1'u\in  E_{\hat{\Gamma}_j}(\Omega_1,U_{1j})$, from Lemma \ref{lemma-20} (i) we have $1=\sigma(e)\sigma(\varphi(e))=\sigma(e'')\sigma(\varphi(e''))$
and so $a_{\omega_1'u}=a_{\omega_1'\varphi(u)}$. It implies that $A_{\Omega_1,U_{2j}}=D_1$. Similar to the analysis described above, we get $A_{\Omega_1,U_{2j}}=-D_1$ if $\sigma(\omega_1u)=-\sigma(\omega_1\varphi(u))$, i.e. $a_{\omega_1u}=-a_{\omega_1\varphi(u)}$.
Furthermore, similar as the above arguments,   for any edge $\omega_2u\in E_{\hat{\Gamma}_j}(\Omega_2, U_{1j})$, we deduce that
$$\left\{\begin{array}{ll}
A_{\Omega_2,U_{2j}}= D_2,& \mbox{if $a_{\omega_2u}=a_{\omega_2\varphi(u)}$,}\\
A_{\Omega_2,U_{2j}}= -D_2,& \mbox{if  $a_{\omega_2u}=-a_{\omega_2\varphi(u)}$. }\end{array}\right.$$
At last,  for any two rotated edges $e=\omega_1 u\in E_{\hat{\Gamma}_j}(\Omega_1,U_{1j})$ and $\tilde{e}=\omega_2\varphi(u)\in E_{\hat{\Gamma}_j}(\Omega_2,U_{2j})$,
we obtain
$\sigma(e)\sigma(\varphi(e))=-\sigma(\tilde{e})\sigma(\varphi(\tilde{e}))$ by Lemma \ref{lemma-20} (ii), which implies that $\sigma(e)=a_{\omega_1u}=a_{\omega_1\varphi(u)}=\sigma(\varphi(e))$ if and only if   $\sigma(\tilde{e})=a_{\omega_2\varphi(u)}=-a_{\omega_2u}
=-\sigma(\varphi(\tilde{e}))$.
 Summarizing above discussions, we claim that (\ref{U-eq-1}) holds.

Therefore,  $\hat{\Gamma}_j\simeq\hat{\Gamma}_j^a$ or $\hat{\Gamma}_j^b$.
We complete this proof.
\end{proof}
\begin{re}
Suppose that $G[F]$ is not connected and  consists of  some connected components: $G[F]=\cup_{i=1}^lG[F_i]$.  Let $G[F_i]=(\Omega_{i1}, \Omega_{i2})$ for $i=1, \ldots, l$. If $G[F_i\cup U_{1j}\cup U_{2j}]$ is disconnected, then there is no any edge between $F_i$ and $U_{1j}\cup U_{2j}$.  If $G[F_i\cup U_{1j}\cup U_{2j}]$ is connected, then, by regarding $G[F_i\cup U_{1j}\cup U_{2j}]$ as $G[F\cup U_{1j}\cup U_{2j}]$ in the proof of Lammas \ref{pro-4-0}, \ref{lemma-20} and \ref{FF-lem-1}, the corresponding results still hold. 
\end{re}

From the proof of Lemma \ref{FF-lem-1}, we see that there is no  reflexive edge in $E_{G_j}(U_{1j}, U_{2j})$ for the reflexive factor $G_j$. It means that the block matrix $A(\hat{\Gamma}_j^a)_{U_{1j}, U_{2j}}$ in \eqref{FF-eq-1} differs from $A(\Gamma_j^a)$ in \eqref{Gamma-j-eq-1}. It implies the following result.

\begin{prop}
Let $G$ be a connected non-bipartite graph with an automorphism $\varphi$ of order two. If there exists a reflexive factor $G_j$ containing a reflexive edge, then $\varphi$ can not be a weak automorphism of any signed graph  on $G$.
\end{prop}

If $\Gamma=\hat{\Gamma}_1=(\hat{G}_1, \sigma)$ (i.e., $\Gamma$ has only one $\varphi$-factor), then Lemma \ref{FF-lem-1} directly gives the following result.
\begin{thm}[Construction Theorem-II]\label{ff-thm-2}
Let  $\Gamma=(G,\sigma)\in\mathbb{FS}_1(G)$. If   $\Gamma=\hat{\Gamma}_1$, then $\Gamma\simeq\hat{\Gamma}^a$ or $\hat{\Gamma}^b$, where $\hat{\Gamma}^a$ and $\hat{\Gamma}^b$ are presented as the form  (\ref{FF-eq-1}).
\end{thm}

Based on  Lemma \ref{FF-lem-1}, the  signed graphs in $ \mathbb{FS}_1(G)$   can be characterized in general form.

\begin{thm}[Construction Theorem-III]\label{ff-thm-3}
$\Gamma=(G,\sigma)\in \mathbb{FS}_1(G)$ if and only if $\Gamma\simeq \tilde{\Gamma}$, where $\tilde{\Gamma}$ is defined by
\begin{equation*}
\begin{scriptsize}
A(\tilde{\Gamma})=\left(\begin{array}{cc|ccccc|ccccc}
    0 & A & D_{11} & D_{11} & \cdots & D_{1m} & D_{1m} & D_{1,m+1} & -D_{1,m+1} & \ldots & D_{1s} & -D_{1s} \\
    A^\top & 0 & D_{21} & -D_{21} & \cdots & D_{2m} & -D_{2m} & D_{2,m+1} & D_{2,m+1} & \cdots & D_{2s} & D_{2s} \\
    \hline
    D_{11}^\top & D_{21}^\top & B_1 & Y_1 &  &  &  &  &  &  &  &  \\
     D_{11}^\top &-D_{21}^\top  & -Y_1 & -B_1 &  &  &  &  &  &  &  &  \\
    \vdots &\vdots  &  &  & \ddots &  &  &  &  &  &  &  \\
     D_{1m}^\top& D_{2m}^\top &  &  &  & B_m & Y_m &  &  &  &  &  \\
     D_{1m}^\top& -D_{2m}^\top &  &  &  & -Y_m & -B_m &  &  &  &  &  \\
    \hline
     D_{1,m+1}^\top& D_{2,m+1}^\top &  &  &  &  &  & B_{m+1} & Y_{m+1} &  &  &  \\
    -D_{1,ms+1}^\top & D_{2,m+1}^\top &  &  &  &  &  & -Y_{m+1} & -B_{m+1} &  &  &  \\
     \vdots& \vdots &  &  &  &  &  &  &  & \ddots &  &  \\
    D_{1s}^\top & D_{2s}^\top &  &  &  &  &  &  &  &  & B_{s} & Y_{s} \\
    -D_{1s}^\top & D_{2s}^\top &  &  &  &  &  &  &  &  & -Y_{s} & -B_{s} \\
  \end{array}\right),
\end{scriptsize}
\end{equation*}
where $A$, $D_{1i}$, $D_{2i}$, $B_i$,  and $Y_i$ are $(0, 1, -1)$-matrices,  $B_i$ is a symmetric matrix with a zero diagonal, and  $Y_i^\top=-Y_i$ for $i=1, \ldots, s$.
\end{thm}
\begin{proof}
By Claim \ref{clm-4-0} (ii), $\Gamma$ has $s$ $\varphi$-factors. By Lemma \ref{FF-lem-1}, we may assume that  $\hat{\Gamma}_j\simeq \hat{\Gamma}_j^a$ for $j=1, \ldots, m$ and $\hat{\Gamma}_j\simeq \hat{\Gamma}_j^b$ for $j=m+1, \ldots, s$. We label the vertices of $\Gamma$  by blocks in the following order: $\Omega_1\cup\Omega_2, U_{11}\cup U_{21},\ldots,U_{1s}\cup U_{2s}$ where $U_{2j}=\varphi(U_{ij})$ for $j=1, \ldots, s$. By  Lemma \ref{FF-lem-1}, we have $\Gamma\simeq \tilde{\Gamma}$.

Conversely, suppose that $\Gamma\simeq \tilde{\Gamma}$. Let
$$P=\left(
      \begin{array}{cc}
        -I &  \\
         & I \\
      \end{array}
    \right)\oplus \underbrace{\left(
      \begin{array}{cc}
        0 & I  \\
         I&  0 \\
      \end{array}
    \right)\oplus \cdots \oplus \left(
      \begin{array}{cc}
        0 & I  \\
         I&  0 \\
      \end{array}
    \right)}_m \oplus \underbrace{\left(
      \begin{array}{cc}
        0 & -I  \\
         -I& 0 \\
      \end{array}
    \right)\oplus\cdots \oplus\left(
      \begin{array}{cc}
        0 & -I  \\
         -I& 0 \\
      \end{array}
    \right)}_{s-m}.$$
It is routine  to verify that $PA(\tilde{\Gamma})P^\top=-A(\tilde{\Gamma})$, that is $\tilde{\Gamma}$ is sign-symmetric. Obviously, $P$ is just a weak automorphism of order 2 of $\tilde{\Gamma}$.
Therefore, $\tilde{\Gamma}\in \mathbb{FS}_1(G)$ and so $\Gamma\in \mathbb{FS}_1(G)$.
\end{proof}

\begin{figure}[h]
  \centering
  \includegraphics[width=6cm]{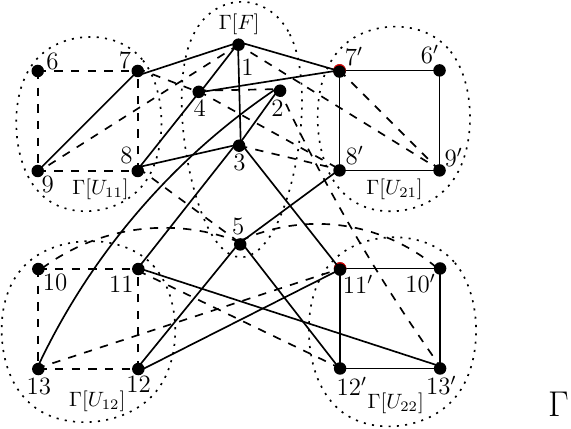}\\
  \caption{\small $\Gamma=(G, \sigma)\in \mathbb{FS}_1(G)$}\label{figure46}
\end{figure}

Now we give an example as the application of Theorem \ref{ff-thm-3}.
\begin{example}
In Fig.$\!$ \ref{figure46},   $\varphi$ is an automorphism of order 2 on $G$ such that $\varphi(i)=i$ for $i=1, \ldots, 5$ and $\varphi(i)=i'$ for $i=6, \ldots, 13$.  $\Gamma[F]=(\Omega_1, \Omega_2)$ is a bipartite signed graph,  where $\Omega_1=\{1,  2\}$ and $\Omega_2=\{3, 4, 5\}$. Let $U_{11}=\{6, 7, 8, 9\}$, $U_{21}=\{6', 7', 8', 9'\}$, $U_{12}=\{10, 11, 12, 13\}$, $U_{22}=\{10', 11', 12', 13'\}$.
We get
$$A(\Gamma)
=\left(
   \begin{array}{cccccc}
0&A_{\Omega_1, \Omega_2} & A_{\Omega_1, U_{11}}& A_{\Omega_1, U_{11}}& A_{\Omega_1, U_{12}} & -A_{\Omega_1, U_{12}} \\ [0.1cm]
A_{\Omega_1, \Omega_2}^\top&0 & A_{\Omega_2,U_{11}}& -A_{\Omega_2, U_{11}}&A_{\Omega_2, U_{12}} &A_{\Omega_2, U_{12}}\\ [0.1cm]
A_{\Omega_1, U_{11}}^\top&A_{\Omega_2, U_{11}}^\top & A(\Gamma[U_{11}])& 0&0&0 \\[0.1cm]
A_{\Omega_1, U_{11}}^\top &-A_{\Omega_2, U_{11}}^\top & 0 & -A(\Gamma[U_{11}])&0&0\\[0.1cm]
A_{\Omega_1, U_{12}}^\top&A_{\Omega_2, U_{12}}^\top &0 & 0&A(\Gamma[U_{12}])&A_{U_{12},U_{22}} \\ [0.1cm]
-A_{\Omega_1, U_{12}}^\top &A_{\Omega_2, U_{12}}^\top & 0 & 0&-A_{U_{12},U_{22}}&-A(\Gamma[U_{12}])
   \end{array}
 \right).$$
 It can be checked that $\varphi$ is a principal weak automorphism of $\Gamma$ of order 2, and so $\Gamma\in \mathbb{FS}_1(G)$. Moreover, we  calculate that
$\mathrm{Spec}(\Gamma)=\{\pm3.705,\pm3.017, \pm2.491, \pm2.373, \pm2.019,$  $\pm1.987, \pm1.349, \pm 0.952, \pm0.744, \pm0.189, 0\}$.
\end{example}

Summarizing the results of Theorems \ref{thm-main-1}, \ref{ff-thm-2} and \ref{ff-thm-3},  we obtain a simple and complete  characterization of all  sign-symmetric signed graphs in  $\mathbb{S}_1(G)$. We have known some  examples of sign-symmetric signed graphs in  $\mathbb{S}_r(G)$ for $r\ge 2$ (for instance, $\Gamma_2\in \mathbb{S}_2(G)$ in Example \ref{ex3}). However, the complete characterization of  sign-symmetric signed graphs in $\mathbb{S}_2(G)$ (and further in $\mathbb{S}_3(G)$ and so on) remains open. We would like to  propose the following problem:
\begin{prob}
 Characterize the sign-symmetric signed graphs in $ \mathbb{S}_r(G)$ for $r\geq2$.
\end{prob}

\vskip 0.1 true cm
\begin{center}{\textbf{Acknowledgments}}
\end{center}

 This project  is supported by the National Natural Science Foundation of China (Nos.
 12001185, 12371353) and the Scientific Research Fund of Hunan Province Education Department (No. 22A0433).

\begin{center}{\textbf{Disclosure statement}}\end{center}
No potential conflict of interest was reported by the authors.


\baselineskip=0.15in

\begin{thebibliography}{99}

\bibitem{Akbaria}
S. Akbari, H.R. Maimani, L. P. Majd, On the spectrum of some signed complete and complete bipartite graphs, \emph{ Filomat}, {\bf 32(17)} (2018), 5817--5826.



\bibitem{Belardo}
F. Belardo, S.M. Cioab\u{a}, J.Koolen, J. Wang, Open problems in the spectral theory of signed
graphs, \emph{Art Discrete Appl. Math.} 1 (2018), \# P2. 10.


\bibitem{Belardo2}
F. Belardo, S.K. Simi\'{c}, On the Laplacian coefficients of signed graphs, \emph{Linear Algebra Appl.} {\bf 475} (2015), 94--113.



\bibitem{Cvetk}
D.M. Cvetkovi\'{c}, P. Rowlinson, S. Simi\'{c}, An Introduction to the Theory of Graph Spectra, Cambridge University Press (2010).


\bibitem{Ghorbani}
E. Ghorbani, W.H. Haemers, H.R. Maimani, L.P. Majd, On sign-symmetric signed graphs, \emph{ARS Math. Contemp.} {\bf 19} (2020), 83--93.





\bibitem{Stanic}
Z. Stani\'{c}, Connected non-complete signed graphs which have symmetric spectrum but are not sign-symmetric, \emph{ Examples and Counterexamples}, {\bf 1} (2021), 100007.


\bibitem{Zaslavsky}
T. Zaslavsky, Matrices in the theory of signed simple graphs, \emph{Advances in Discrete Mathematics and Applications}: Mysore 2008, Ramanujan Math. Soc., Mysore, 2010, pp. 207--229.

\bibitem{Zaslavsky1}
T. Zaslavsky, Signed graphs, \emph{Discrete Appl. Math.} {\bf 4} (1982), 47--74.
\end{thebibliography}
\end{document}